\documentclass{article}

\usepackage{amsmath, amssymb, amsthm}

\newtheorem{thm}{Theorem}
\newtheorem{lem}{Lemma}

\newcommand{\dlim}{\displaystyle\lim}
\newcommand{\ds}{\displaystyle}

\title{A new approach to the results of K\"ovari, S\'os, and Tur\'an concerning rectangle-free subsets of the grid}
\author{Jeremy F. Alm\\ Illinois College\\ Jacksonville, IL 62650\\ \texttt{alm.academic@gmail.com} \and Jacob Manske\footnote{Corresponding author.}\\ Texas State University\\ San Marcos, TX 78666\\ \texttt{jmanske@txstate.edu}}
\begin{document}
\maketitle

%%%%%%%%%%%%%%%%
%%% ABSTRACT %%%
%%%%%%%%%%%%%%%%
\begin{abstract}
For positive integers $m$ and $n$, define $f(m,n)$ to be the smallest integer such that any subset $A$ of the $m \times n$ integer grid with $|A| \geq f(m,n)$ contains a rectangle; that is, there are $x\in [m]$ and $y \in [n]$ and $d_{1},d_{2} \in \mathbb{Z}^{+}$ such that all four points $(x,y)$, $(x+d_{1},y)$, $(x,y+d_{2})$, and $(x+d_{1},y+d_{2})$ are contained in $A$. In~\cite{kovarisosturan}, K\"ovari, S\'os, and Tur\'an showed that $\dlim_{k \to \infty}\dfrac{f(k,k)}{k^{3/2}} = 1$. They also showed that $f(p^{2},p^{2}+p) = p^{2}(p+1)+1$ whenever $p$ is a prime number. We recover their asymptotic result and strengthen the second, providing cleaner proofs which exploit a connection to projective planes, first noticed by Mendelsohn in~\cite{mendelsohn87}. We also provide an explicit lower bound for $f(k,k)$ which holds for all $k$.

\end{abstract}

%%%%%%%%%%%%%%%%%%%%
%%% INTRODUCTION %%%
%%%%%%%%%%%%%%%%%%%%
\section{Introduction and motivation}

For a positive integer $n$, let $[n] = \left\{1,2,\ldots,n\right\}$. For $m,n \in \mathbb{Z}^{+}$, define $f(m,n)$ to be the least integer such that if $A \subseteq [m]\times [n]$ with $|A| \geq f(m,n)$, then $A$ contains a rectangle; that is, there is $x \in [m], y \in [n]$, and $d_{1},d_{2} \in \mathbb{Z}^{+}$ such that all four points $(x,y)$, $(x+d_{1},y)$, $(x,y+d_{2})$, and $(x+d_{1},y+d_{2})$ are contained in $A$. For ease in notation, let $f(k) = f(k,k)$. For $c \in \mathbb{Z}^{+}$, a \emph{$c$-coloring} of a set $S$ is a surjective map $\chi: S \rightarrow [c]$. If $\chi$ is constant on a set $A \subset S$, we say that $A$ is \emph{monochromatic}.

We will write $g(k) \sim h(k)$ to mean that functions $g$ and $h$ are \emph{asymptotically equal}; that is, $\dlim_{k \to \infty}\dfrac{g(k)}{h(k)} = 1$. Also, notice that $f(m,n) = f(n,m)$ for any choice of $n$ and $m$. 

The problem of finding bounds or exact values of $f(m,n)$ finds its roots in the famous theorem of van der Waerden from~\cite{vdwtheorem}, which states that given any positive integers $c$ and $d$, there exists an integer $N$ such that any $c$-coloring of $[N]$ contains a monochromatic arithmetic progression of length $d$. Szemer\'edi proved a density version of this theorem in~\cite{szemeredi75}, using the now well-known Regularity Lemma. Progress in this area is still being made. For instance, in~\cite{axenovichmanske08}, Axenovich and the second author try to find the smallest $k$ so that in any $2$-coloring of $[k]\times[k]$ there is a monochromatic \emph{square}; i.e., a rectangle with $d_{1}=d_{2}$. While the upper bounds are enormous, they proved $k \geq 13$; in~\cite{bacheli09}, Bacher and Eliahou show that $k=15$. In~\cite{gasarchandco}, the authors are interested in finding OBS$_{c}$, which is the collection of $[m]\times[n]$ grids which cannot be colored in $c$ colors without a monochromatic rectangle, but every proper subgrid can be; see also~\cite{cooperfennerpurewal}. 
For a more complete survey on van der Waerden type problems, see~\cite{ramseybook}.

Zarankiewicz introduced the problem of finding $f(m,n)$ in~\cite{zarankiewicz} using the language of minors of (0,1)-matrices. In~\cite{kovarisosturan}, K\"ovari, S\'os, and Tur\'an show that $f(k) \sim k^{3/2}$ and that whenever $p$ is a prime number, we have $f(p^{2}+p,p^{2}) = p^{2}(p+1) + 1$. In this manuscript, we will recover this asymptotic result and strengthen the second result.

In~\cite{reiman58}, Reiman achieved the bound of 
\begin{equation}
f(m,n) \leq \dfrac{1}{2}\left(m + \sqrt{m^{2}+4mn(n-1)}\right) + 1.\label{reimanineq}
\end{equation}
Notice that by setting $m = p^{2}+p$ and $n=p^{2}$, the right hand side of (\ref{reimanineq}) becomes $p^{2}(p+1) + 1$, so the result of K\"ovari, S\'os, and Tur\'an implies that the inequality is sharp. Reiman showed equality in (\ref{reimanineq}) in the case that $m=n=q^{2}+q+1$, provided $q$ is a prime power. In~\cite{mendelsohn87}, Mendelsohn recovers and strengthens the equality result of Reiman by noticing the connection of the Zarankiewicz problem to projective planes.

A $k \times k$ $(0,1)$-matrix $A$ corresponds to a subset $S_{A} \subset [k]\times [k]$ by 
\begin{center}
$(i,j) \in S$ if and only if the $(i,j)$ entry of $A$ is $1$.
\end{center}
Notice that the set $S_{A}$ contains a rectangle if and only if the matrix $A^{T}A$ has an entry off the main diagonal which is not equal to $0$ or $1$. Also notice that $tr(A^{T}A) = |S_{A}|$.

Such $(0,1)$-matrices arise in the study of projective planes. A projective plane of order $n$ is an incidence structure consisting of $n^{2}+n+1$ points and $n^{2}+n+1$ lines such that 
\begin{enumerate}
\item
any two distinct points lie on exactly one line;
\item
any two distinct lines intersect in exactly one point;
\item
each line contains exactly $n+1$ points; and
\item
there is a set of 4 points such that no 3 of these points lie on the same line.
\end{enumerate}

It is not known for which positive integers $n$ there exists a projective plane of order $n$; projective planes have been constructed for all prime-power orders, but for no others. In the well-known paper~\cite{bruckryser}, Bruck and Ryser show that if the square-free part of $n$ is divisible by a prime of the form $4k+3$, and if $n$ is congruent to $1$ or $2$ modulo $4$, then there is no projective plane of order $n$; see also~\cite{chowlapp}. More recently, the authors in~\cite{dwlv} draw a connection between the existence of projective planes of order greater than or equal to 157 and the number of cycles in $n \times n$ bipartite graphs of girth at least 6. In 1989, a computer search conducted by the authors in~\cite{nopporder10} showed that there is no projective plane of order 10. The smallest order for which it is still not known whether there is a projective plane is $12$, although the results in~\cite{prince99,suetake04,prince04,suetakeakiyama08,suetakeakiyama09} suggest that there is no such structure.

Next we state a lemma which appears in~\cite{mendelsohn87} connecting projective planes to the Zarankiewicz problem.
\begin{lem}
If $n$ is a positive integer such that there exists a projective plane of order $n$, then $f(n^{2}+n+1) = (n+1)(n^{2}+n+1) + 1$.
\label{mendellem}
\end{lem}

We will include a proof of Lemma \ref{mendellem} both for completeness and since we will reference the lower bound construction in the proof of Theorem \ref{affinethm}.

\begin{proof}[Proof of Lemma \ref{mendellem}] Let $n$ be a positive integer such that there is a projective plane of that order. 
For ease in notation, set $N = n^{2}+n+1$.
First we will show that $f(N) \geq (n+1)N + 1$.

 We begin by constructing a $N\times N$ $(0,1)$-matrix $A$. There exists a projective plane $P$ of order $n$; so let $A$ be the $N \times N$ matrix whose rows correspond to the points of $P$ and whose columns correspond to the lines of $P$ where the $(i,j)$ entry of $A$ is equal to $1$ if and only if the point indexed by $i$ lies on the line indexed by $j$. Since any two distinct lines have exactly one point in common, the scalar product of any two distinct columns must be 1; hence, $S_{A}$ does not contain a rectangle. Since each line contains exactly $(n+1)$ points, $|S_{A}| = tr(A^{T}A) = (n+1)N$, so $f(N) \geq (n+1)N + 1$.

Now, suppose $A$ is any $N \times N$ $(0,1)$-matrix with $(n+1)N + 1$ nonzero entries, and let $a_{i}$ denote the number of $1$s in row $i$. 
The number of pairs of 1s in row $i$ is $\ds\binom{a_{i}}{2}$, so the total number of pairs of 1s from each row is $\ds\sum_{i=1}^{N}\binom{a_{i}}{2}$.
The number of pairs of distinct column indices is $\ds\binom{N}{2}$.
If $\ds\sum_{i=1}^{N}\ds\binom{a_{i}}{2} > \ds\binom{N}{2}$, the pigeonhole principle implies that there is a pair of column indices such that there are two distinct rows which have 1s in both of those columns; i.e., $S_{A}$ contains a rectangle.

To see that $\ds\sum_{i=1}^{N}\ds\binom{a_{i}}{2} > \ds\binom{N}{2}$, recall that the Cauchy-Schwarz inequality gives
\begin{equation}\label{CSineq}
\left(\ds\sum_{i=1}^{N} a_{i}\right)^{2}\leq\ds\sum_{i=1}^{N}a_{i}^{2}\ds\sum_{i=1}^{N}1^{2}.
\end{equation}
Since $\ds\sum_{i=1}^{N}a_{i} = (n+1)N + 1$ by assumption, the bound in (\ref{CSineq}) gives
\begin{equation}\label{CSineq2}
(n+1)^{2}N + 2(n+1) + \dfrac{1}{N} \leq \ds\sum_{i=1}a_{i}^{2}.
\end{equation}
Since $\ds\sum_{i=1}^{N}a_{i}^{2} = \ds\sum_{i=1}^{N}a_{i}(a_{i}-1) + \ds\sum_{i=1}^{N}a_{i} = 2\ds\sum_{i=1}^{N}\ds\binom{a_{i}}{2} +  (n+1)N + 1$, inequality (\ref{CSineq2}) gives
\begin{equation}\label{CSineq3}
N\left((n+1)^{2} - (n+1)\right)  + 2(n+1) + \dfrac{1}{N} - 1 \leq 2\ds\sum_{i=1}^{N}\ds\binom{a_{i}}{2}.
\end{equation}
Since $(n+1)^{2} - (n+1) = n^{2} + n + 1 - 1 = N-1$, inequality (\ref{CSineq3}) can be rewritten as
\begin{equation}\label{CSineq4}
\dfrac{N(N-1)}{2} + n +\dfrac{1}{N} +\dfrac{1}{2} \leq \ds\sum_{i=1}^{N}\ds\binom{a_{i}}{2},
\end{equation}
and since $n > 0$, the left hand side of (\ref{CSineq4}) is bound from below by $\ds\binom{N}{2}$, as desired.
\end{proof}

It is interesting to note that we have equality in (\ref{CSineq}) just in case all of the $a_{i}$ are equal; that is, each row and column contain the same number of 1s.

%%%%%%%%%%%%%%%%%%%%%%%%%%%%%
%%% STATEMENTS OF REUSLTS %%%
%%%%%%%%%%%%%%%%%%%%%%%%%%%%%

\section{Main results}

Our main lemma is below, a useful proposition for dealing with asymptotic behavior of functions when some explicit values of the functions are known. A version of this lemma is used in~\cite{kovarisosturan}, but it is neither proved nor explicitly stated.

\begin{lem} Suppose $g$ and $h$ are monotonically increasing functions. If $a_{n}$ is a strictly increasing sequence of positive integers such that 
\begin{enumerate}
\item 
$\dlim_{n \to \infty}\dfrac{a_{n+1}}{a_{n}} = 1$;
\item 
$\dlim_{n \to \infty}\dfrac{h\left(a_{n+1}\right)}{h\left(a_{n}\right)}=1$; and 
\item
$g(a_{n}) = h(a_{n})$ for all $n$,
\end{enumerate} then $g \sim h$.

\label{mainlem}
\end{lem}

%%%%%%%%%%%%%%%%%%%%%%%%%%%%%%%%%
%%% STATEMENT OF MAIN THEOREM %%%
%%%%%%%%%%%%%%%%%%%%%%%%%%%%%%%%%

Theorem \ref{mainthm} recovers the asymptotic result of K\"ovari, S\'os, and Tur\'an. Theorem \ref{affinethm} strengthens another of their results. The proofs exploit the connection to projective planes, cleaning up the arguments found in~\cite{kovarisosturan}. 
Theorem \ref{lowerbound} is an explicit lower bound for $f(k)$, which holds for all $k$.

\begin{thm} $f(k) \sim k^{3/2}$.

\label{mainthm}
\end{thm}

\begin{thm}\label{affinethm}
Let $n$ be a positive integer. If there is a projective plane of order $n$, then $f(n^{2},n^{2}+n) = n^{2}\left(n+1\right) + 1$.

\end{thm}

\begin{thm}\label{lowerbound}
If $k \in \mathbb{Z}$ with $k \geq 3$, then $f(k) \geq \dfrac{1}{16}\left((k+4)\sqrt{4k-3} + 5k + 22\right)$.

\end{thm}

\section{Proof of Lemma \ref{mainlem}}

Now we prove Lemma \ref{mainlem}.

\begin{proof} Let $g$ and $h$ be monotonically increasing functions. Suppose $a_{n}$ is a strictly increasing sequence of positive integers such that $\dlim_{n \to \infty}\dfrac{h\left(a_{n+1}\right)}{h\left(a_{n}\right)}=1$ and that $g\left(a_{n}\right) = h\left(a_{n}\right)$ for all $n$. 
Let $\varepsilon > 0$. 
Choose $N$ so that 
\begin{equation}\label{initialineq}
\left|\dfrac{h(a_{n+1})}{h(a_{n})} - 1\right| < \varepsilon\ \text{and}\ \left|\dfrac{h(a_{n})}{h(a_{n+1})} -1\right|< \varepsilon
\end{equation} whenever $n > N$. 
Next, choose $m$ large enough so that for some $n > N$, we have $a_{n} \leq m \leq a_{n+1}$. 
Since $g$ is increasing and $g$ and $h$ agree on the sequence $a_{n}$, we have 
\begin{equation}\label{ineqs:1}
h(a_{n}) = g(a_{n}) \leq g(m) \leq g(a_{n+1}) = h(a_{n+1}).
\end{equation}
Since $h$ is monotone increasing, $h(a_{n}) \leq h(m) \leq h(a_{n+1})$, so we may transform (\ref{ineqs:1}) into
\begin{equation}\label{ineqs:2}
\dfrac{h(a_{n})}{h(a_{n+1})} \leq \dfrac{g(m)}{h(m)} \leq \dfrac{h(a_{n+1})}{h(a_{n})}.
\end{equation}
Subtracting 1 from every term in (\ref{ineqs:2}) and taking absolute values gives that either 
\[\left|\dfrac{g(m)}{h(m)} - 1\right| \leq \left|\dfrac{h(a_{n+1})}{h(a_{n})} - 1\right|\ \text{or}\ \left|\dfrac{g(m)}{h(m)} - 1\right| \leq \left|\dfrac{h(a_{n})}{h(a_{n+1})} - 1\right|.\]
Without loss of generality, say $\left|\dfrac{g(m)}{h(m)} - 1\right| \leq \left|\dfrac{h(a_{n+1})}{h(a_{n})} - 1\right|$.
By (\ref{initialineq}), we have
\[\left|\dfrac{g(m)}{h(m)} - 1 \right| < \varepsilon,\]
so $\dfrac{g}{h} \to 1$ and $g \sim h$, as desired.
\end{proof}

%%%%%%%%%%%%%%%%%%%%%%%%%%%%%%%%%%
%%% PROOF OF ASYMPTOTIC RESULT %%%
%%%%%%%%%%%%%%%%%%%%%%%%%%%%%%%%%%

\section{Proof of Theorem \ref{mainthm}}\label{mainthmproof}

Now we prove Theorem \ref{mainthm}.
\begin{proof} For a positive integer $k$, set \[h(k) = \left(\sqrt{k - \dfrac{3}{4}} + \dfrac{1}{2}\right)k + 1. \]
Notice that $h(k) \sim k^{3/2}$ and that $h(n^{2}+n+1) = (n+1)(n^{2}+n+1)+1$, so by Lemma \ref{mendellem}, we have $f(n^{2}+n+1) = h(n^{2}+n+1)$ whenever there is a projective plane of order $n$. Since there a projective plane of order $p$ for every prime $p$, we have that $f$ and $h$ agree on an infinite sequence of integers $a_{n}$ for which $\dfrac{a_{n+1}}{a_{n}} \rightarrow 1$ (see \cite{selbergprimes,erdosprimes}). Notice that $\dfrac{h\left(a_{n+1}\right)}{h\left(a_{n}\right)}\rightarrow 1$, so we may apply Lemma \ref{mainlem} to achieve $f \sim h$, and thus $f \sim k^{3/2}$, as desired.
\end{proof}

%%%%%%%%%%%%%%%%%%%%%%%%%%%%%%%%%%%%%
%%% PROOF OF AFFINE PLANE THEOREM %%%
%%%%%%%%%%%%%%%%%%%%%%%%%%%%%%%%%%%%%
\section{Proof of Theorem \ref{affinethm}}

\begin{proof}
Let $n$ be a positive integer such that there is a projective plane of order $n$. Set $N = n^{2} + n + 1$. As in the proof of Lemma \ref{mendellem}, we can construct an $N \times N$ matrix $A$ such that $tr\left(A^{T}A\right) = (n+1)N$ and that $A^{T}A$ has only 1s off the main diagonal; hence, the corresponding subset $S_{A}$ of the $N \times N$ grid has no rectangle.

To construct an $n^{2}\times \left(n^{2}+n\right)$ matrix $B$ from $A$, we delete the first column of $A$ along with all rows having a 1 in the first column. Since each row and column of $A$ contains exactly $n+1$ nonzero entries, we have deleted $n+1$ rows and $1$ column. The resulting matrix $B$ is thus an $n^{2}\times \left(n^{2}+n\right)$ matrix. Since $A^{T}A$ has no entries off the main diagonal greater than 1, $B^{T}B$ has no entries off the main diagonal greater than 1. Since we have deleted $\left(n+1\right)^{2}$ nonzero entries from $A$, we have that 
\[\left|S_{B}\right| = (n+1)N - \left(n+1\right)^{2} = \left(n+1\right)\left(n^{2}+ n + 1\right) - \left(n+1\right)^{2} = n^{2}(n+1),\]
so $f\left(n^{2}, n^{2}+n\right) \geq n^{2}\left(n + 1\right)+1.$

Using the inequality from Reiman (\ref{reimanineq}), 
\[f\left(n^{2},n^{2}+n\right) \leq n^{2}\left(n+1\right)+1,\] and hence $f\left(n^{2},n^{2}+n\right) = n^{2}\left(n+1\right)+1$, as desired.
\end{proof}

The structure obtained by taking a projective plane and deleting a line together with all of the points on that line is called an \emph{affine plane}. Our result is stronger than that of the authors in~\cite{kovarisosturan}, since we need only that there is a projective plane of order $n$, not that $n$ is a prime number.

%%%%%%%%%%%%%%%%%%%%%%%%%%%%%%%%%%%%%
%%% PROOF OF EXPLICIT LOWER BOUND %%%
%%%%%%%%%%%%%%%%%%%%%%%%%%%%%%%%%%%%%

\section{Proof of Theorem \ref{lowerbound}}

\begin{proof}
Suppose $k$ is an integer with $k \geq 3$. There exists a nonnegative integer $\alpha$ such that
\begin{equation}\label{alphabound}
2^{2\alpha} + 2^{\alpha}+1 \leq k \leq 2^{2\alpha + 2} + 2^{\alpha + 1} + 1.
\end{equation}
By focusing on the upper bound from (\ref{alphabound}), this gives $k \leq \left(2^{\alpha+1} + 1/2\right)^{2} + 3/4$, or 
\begin{equation}\label{LB1}
\dfrac{\sqrt{k - 3/4} - 1/2}{2} \leq 2^{\alpha}.
\end{equation}
Let $g(n) = (n+1)(n^{2} + n + 1) + 1$, and let $h(k) = \dfrac{\sqrt{k - 3/4} - 1/2}{2}$. Since $g$ is an increasing function, inequality (\ref{LB1}) gives
\begin{equation}\label{LB2}
g\left(h(k)\right) \leq g\left(2^{\alpha}\right).
\end{equation}
By Lemma \ref{mendellem}, we have $g(n) = f(n^{2} + n + 1)$ whenever there exists a projective plane of order $n$. Since there is a projective plane of any prime power order, (\ref{LB2}) gives
\begin{equation}\label{LB3}
g\left(h(k)\right) \leq f\left(2^{2\alpha} + 2^{\alpha} + 1\right).
\end{equation}
But since $f$ is increasing, the lower bound in (\ref{alphabound}) gives $g\left(h(k)\right) \leq f(k)$, and since $g\left(h(k)\right) = \dfrac{1}{16}\left((k+4)\sqrt{4k-3} + 5k + 22\right)$, we have the desired result.

We also note that while $g\left(h(k)\right) \sim \dfrac{1}{8}k^{3/2}$, which is worse than the result in Theorem \ref{mainthm}, this lower bound holds for every choice of $k$, and not just those $k$ for which there exists a projective plane of order $k$.
\end{proof}

%%%%%%%%%%%%%%%%%%%%%%%%
%%% FURTHER RESEARCH %%%
%%%%%%%%%%%%%%%%%%%%%%%%

\section{Further Research}

Trying to find the exact value of $f(m,n)$ without conditions on $m$ and $n$ (that is, removing the extra hypotheses from the results in~\cite{kovarisosturan}) would be attractive, although this problem has been open for years, and likely requires a new idea.

The next attractive direction is to take the approach of the authors in~\cite{gasarchandco}, and consider colorings of rectangular grids.

Recall that OBS$_{c}$ is the collection of $[m]\times[n]$ grids which cannot be colored in $c$ colors without a monochromatic rectangle, but every proper subgrid can be. An open problem from~\cite{gasarchandco} is the \emph{rectangle-free conjecture}: if there exists a rectangle-free subset of $[m]\times[n]$ of size $\left\lceil mn/c\right\rceil$, then it is possible to color $[m] \times [n]$ in $c$ colors so there is no monochromatic rectangle. Since the authors in~\cite{gasarchandco} have theorems which depend on the rectangle-free conjecture, resolving this conjecture either in the affirmative or the negative would result in progress for obtaining $|\text{OBS}_{c}|$ or even OBS$_{c}$.

\noindent
{\bf Acknowledgments.} The authors wish to thank Jim Marshall for carefully reading a draft of this manuscript. The authors also thank the anonymous referees for their helpful comments made toward improving the paper.

\bibliographystyle{plain}

\begin{thebibliography}{10}

\bibitem{suetakeakiyama08}
K.~Akiyama and C.~Suetake.
\newblock The nonexistence of projective planes of order 12 with a collineation
  group of order 8.
\newblock {\em J. Combin. Des.}, 16(5):411--430, 2008.

\bibitem{suetakeakiyama09}
K.~Akiyama and C.~Suetake.
\newblock On projective planes of order 12 with a collineation group of order
  9.
\newblock {\em Australas. J. Combin.}, 43:133--162, 2009.

\bibitem{axenovichmanske08}
M.~Axenovich and J.~Manske.
\newblock On monochromatic subsets of a rectangular grid.
\newblock {\em Integers}, 8:A21, 14, 2008.

\bibitem{bacheli09}
R.~Bacher and S.~Eliahou.
\newblock Extremal binary matrices without constant 2-squares.
\newblock {\em J. Comb.}, 1(1, [ISSN 1097-959X on cover]):77--100, 2010.

\bibitem{bruckryser}
R.~H. Bruck and H.~J. Ryser.
\newblock The nonexistence of certain finite projective planes.
\newblock {\em Canadian J. Math.}, 1:88--93, 1949.

\bibitem{chowlapp}
S.~Chowla and H.~J. Ryser.
\newblock Combinatorial problems.
\newblock {\em Canadian J. Math.}, 2:93--99, 1950.

\bibitem{cooperfennerpurewal}
J.~Cooper, S.~Fenner, and S.~Purewal.
\newblock Monochromatic boxes in colored grids.
\newblock {\em SIAM J. Discrete Math.}, 25(3):1054--1068, 2011.

\bibitem{dwlv}
S.~De~Winter, F.~Lazebnik, and J.~Verstra{\"e}te.
\newblock An extremal characterization of projective planes.
\newblock {\em Electron. J. Combin.}, 15(1):Research Paper 143, 13, 2008.

\bibitem{erdosprimes}
P.~Erd{\"o}s.
\newblock On a new method in elementary number theory which leads to an
  elementary proof of the prime number theorem.
\newblock {\em Proc. Nat. Acad. Sci. U. S. A.}, 35:374--384, 1949.

\bibitem{gasarchandco}
S.~Fenner, W~Gasarch, C.~Glover, and S.~Purewal.
\newblock Rectangle free coloring of grids.
\newblock {\em {\tt arXiv:1005.3750 [math.CO]}}, 2010.

\bibitem{ramseybook}
R.~L. Graham, B.~L. Rothschild, and J.~H. Spencer.
\newblock {\em Ramsey theory}.
\newblock Wiley-Interscience Series in Discrete Mathematics and Optimization.
  John Wiley \& Sons Inc., New York, second edition, 1990.
\newblock A Wiley-Interscience Publication.

\bibitem{kovarisosturan}
T.~K{\"o}vari, V.~T. S{\'o}s, and P.~Tur{\'a}n.
\newblock On a problem of {K}. {Z}arankiewicz.
\newblock {\em Colloquium Math.}, 3:50--57, 1954.

\bibitem{nopporder10}
C.~W.~H. Lam, L.~Thiel, and S.~Swiercz.
\newblock The nonexistence of finite projective planes of order {$10$}.
\newblock {\em Canad. J. Math.}, 41(6):1117--1123, 1989.

\bibitem{mendelsohn87}
N.~S. Mendelsohn.
\newblock Packing a square lattice with a rectangle-free set of points.
\newblock {\em Math. Mag.}, 60(4):229--233, 1987.

\bibitem{prince99}
A.~R. Prince.
\newblock Projective planes of order {$12$} and {${\rm PG}(3,3)$}.
\newblock {\em Discrete Math.}, 208/209:477--483, 1999.
\newblock Combinatorics (Assisi, 1996).

\bibitem{prince04}
A.~R. Prince.
\newblock Ovals in finite projective planes via the representation theory of
  the symmetric group.
\newblock In {\em Finite groups 2003}, pages 283--290. Walter de Gruyter GmbH
  \& Co. KG, Berlin, 2004.

\bibitem{reiman58}
I.~Reiman.
\newblock \"{U}ber ein {P}roblem von {K}. {Z}arankiewicz.
\newblock {\em Acta. Math. Acad. Sci. Hungar.}, 9:269--273, 1958.

\bibitem{selbergprimes}
A.~Selberg.
\newblock An elementary proof of the prime-number theorem.
\newblock {\em Ann. of Math. (2)}, 50:305--313, 1949.

\bibitem{suetake04}
C.~Suetake.
\newblock The nonexistence of projective planes of order 12 with a collineation
  group of order 16.
\newblock {\em J. Combin. Theory Ser. A}, 107(1):21--48, 2004.

\bibitem{szemeredi75}
E.~Szemer{\'e}di.
\newblock On sets of integers containing no {$k$} elements in arithmetic
  progression.
\newblock {\em Acta Arith.}, 27:199--245, 1975.
\newblock Collection of articles in memory of Juri\u\i\ Vladimirovi\v c Linnik.

\bibitem{vdwtheorem}
B.L. van~der Waerden.
\newblock Beweis einer {B}audetchen {V}ermutung.
\newblock {\em Nieuw Arch. Wiskunde}, 15:212--216, 1927.

\bibitem{zarankiewicz}
K.~Zarankiewicz.
\newblock Problem {P}101.
\newblock {\em Colloq. Math.}, 3:19--30, 1954.

\end{thebibliography}

\end{document}